\documentclass[12pt]{amsart}
\usepackage{amssymb}

\usepackage{t1enc}
\usepackage[latin2]{inputenc}
\usepackage{verbatim}
\usepackage{amsmath,amsfonts,amssymb,amsthm}
\usepackage[mathcal]{eucal}
\usepackage{enumerate}

\newtheorem{theorem}{Theorem}

\newtheorem{lemma}{Lemma}

\newtheorem{corollary}{Corollary}

 \DeclareMathOperator{\mes}{mes}
 \DeclareMathOperator{\supp}{supp}

\begin{document}
\author{K\'aroly Nagy, Mohamed Salim}
\title[Restricted summability of ...]{Restricted summability of the multi-dimensional Ces\`aro means of  Walsh-Kaczmarz-Fourier series }
\thanks{Research supported by project T\'AMOP-4.2.2.A-11/1/KONV-2012-0051, GINOP-2.2.1-15-2017-00055 and UAEU UPAR  2017 Grant G00002599.}

\address{K. Nagy, Institute of Mathematics and Computer Sciences, University of Ny\'\i
regyh\'aza, P.O. Box 166, Ny\'\i regyh\'aza, H-4400 Hungary }
\email{nagy.karoly@nye.hu}

\address{M. Salim, Department of Mathematical Sciences, UAEU, Al Ain, United Arab Emirates}
\email{msalim@uaeu.ac.ae}

\date{}

\begin{abstract} 
The properties of the maximal operator of the $(C,\alpha)$-means ($\alpha=(\alpha_1,\ldots,\alpha_d)$) of the  multi-dimensional Walsh-Kaczmarz-Fourier series are discussed,  where the set of indices is inside a cone-like set. We prove that the maximal operator  is bounded from $H_p^\gamma$ 
to $L_p$ for $p_0< p $ ($p_0=\max{1/(1+\alpha_k): k=1,\ldots, d}$) and is of weak type $(1,1)$. As a corollary we get the 
theorem of Simon \cite{S1} on the a.e. convergence of cone-restricted two-dimensional  Fej\'er means of integrable functions. 
At the endpoint case $p=p_0$, we show that the maximal operator $\sigma^{\kappa,\alpha,*}_L$ is not bounded from the Hardy space $H_{p_0}^\gamma$ to the space $L_{p_0}$. 
\end{abstract}
\maketitle

\section{Definitions and notation}

Now, we give a brief introduction to the theory of dyadic analysis (for more details see \cite{AVDR,SWSP}).
Let $\mathbb{P}$ denote the set of positive integers, $\mathbb{N:=P}\cup \{ 0\}.$ Denote ${\mathbb Z}_{2}$ the discrete cyclic group of order 2, that is ${\mathbb Z}_{2}$ has two elements 0 and 1, the group operation is the modulo 2 addition. The topology is given by that  every subset is open. The Haar measure on ${\mathbb Z}_{2}$ is given
such that the measure of a singleton is 1/2, that is, $\mu(\{ 0\})=\mu(\{ 1\})=1/2$. Let $G$ be the
complete direct product of  countable infinite copies of the
compact groups ${\mathbb Z}_{2}.$ The elements of $G$ are sequences of the form
$x=\left( x_{0},x_{1},...,x_{k},...\right) $ with coordinates $x_{k}\in
\{0,1\}\left( k\in \mathbb{N}\right) .$ The group operation on $G$
is the coordinate-wise addition, the measure (denoted
by $\mu $) is the product measure and the topology is the  product 
topology. The compact Abelian group $G$ is called the Walsh group.
A base for the neighbourhoods of $G$ can be given by 
\begin{equation*}
I_{0}\left( x\right) :=G,\quad I_{n}\left( x\right)
:=I_{n}\left( x_{0},\ldots,x_{n-1}\right)
:=\left\{ y\in
G:\,y=\left( x_{0},\ldots,x_{n-1},y_{n},y_{n+1},\ldots \right) \right\} ,
\end{equation*}
$\left( x\in G,n\in \mathbb{N}\right) $, $I_n(x)$ are called  dyadic intervals.
Let $0=\left( 0:i\in \mathbb{N}\right) \in G$ denote the null element of
$G,$ and for the simplicity we write $I_{n}:=I_{n}\left( 0\right)$ $\left( n\in
\mathbb{N}\right) .$
Set $e_{n}:=\left( 0,\ldots,0,1,0,\ldots \right) \in G,$ the $n$th coordinate of which is 1 and the rest are zeros
$\left( n\in \mathbb{N}\right) .$

Let $r_k$ denote the $k$-th Rademacher function, it is defined by 
\begin{equation*}
r_{k}\left( x\right) :=\left( -1\right) ^{x_{k}}
\end{equation*}
($k\in \mathbb{N}$, $x\in G$). 

The Walsh-Paley system is defined as the product system of Rademacher functions. Now, we give more details. If $n\in \mathbb{N}$, then $n$ can be expressed in the form  $n=\sum\limits_{i=0}^{\infty }n_{i}2^{i}$, where $n_{i}\in
\{0,1\}\quad\left( i\in \mathbb{N}\right) $.
Define the order of a natural number $n$ by 
$\left| n\right| :=\max\{j\in \mathbb{N:}n_{j}\neq 0\}$, that is
$2^{\left| n\right|}\leq n<2^{\left| n\right| +1}.$

The Walsh-Paley system is
\begin{equation*}
w_{n}\left( x\right) :=\prod\limits_{k=0}^{\infty }\left(
r_{k}\left( x\right) \right) ^{n_{k}}=r_{\left| n\right| }\left(
x\right) \left( -1\right) ^{\sum\limits_{k=0}^{\left| n\right|
-1}n_{k}x_{k}}\quad\left( x\in G,n\in \mathbb{P}\right) .
\end{equation*}

The Walsh-Kaczmarz functions are defined by $\kappa_0=1$ and for $n\ge 1$
$$
\kappa_n(x):=r_{\vert n\vert }(x)\prod_{k=0}^{\vert n\vert -1}
(r_{\vert n\vert -1-k}(x))^{n_k}
=r_{\vert n\vert }(x)(-1)^{\sum_{k=0}^{\vert n\vert -1}
n_kx_{\vert n\vert -1-k}}.
$$
The set of Walsh-Kaczmarz functions and the set of Walsh-Paley functions are equal in dyadic blocks. Namely,
$$
\{\kappa_n: 2^k\leq n<2^{k+1}\} =\{ w_n:2^k\leq n<2^{k+1}\}
$$
for all $k\in {\mathbb P}$. Moreover, $ \kappa_0=w_0.$

V.A. Skvortsov (see \cite{Sk1}) gave a relation between the
Walsh-Kaczmarz functions and the Walsh-Paley functions. Namely, he defined a  transformation $\tau_A\colon G\to G$ 
$$
\tau_A(x):=(x_{A-1},x_{A-2},...,x_1,x_0,x_A,x_{A+1},...)$$
for
$A\in
{\mathbb N}.$
By the definition of $\tau_A$, we have the following connection 
$$
\kappa_n(x)=r_{\vert n\vert }(x)w_{n-2^{|n|}}(\tau_{\vert n\vert}(x)) \quad
(n\in {\mathbb N},x\in G).
$$

Let $0<\alpha$ and let 
$$
A_j^\alpha :={j+\alpha \choose j}=\frac{(\alpha +1)(\alpha +2)\ldots (\alpha +j)}{j!} \quad (j\in {\mathbb N}; \alpha\neq -1,-2,\ldots ).
$$
It is known that $A_j^\alpha \sim O(j^\alpha)$ ($j\in {\mathbb N}$) (see Zygmund \cite{Zyg}). 
The one-dimensional Dirichlet kernels and  Ces\'aro kernels are defined by
$$
 D_{n}^\psi  := \sum _{k=0}^{n-1} \psi_k, \quad K_{n}^{\psi,\alpha}(x):=\frac{1}{A_n^\alpha}\sum_{k=0}^{n}A_{n-k}^{\alpha-1} D_{k}^\psi(x),
$$ 
for $\psi_n =w_n$ ($n\in{\mathbb P}$)   or $\psi_n=\kappa_n$  ($n\in {\mathbb
P}$), $D_0^\psi :=0.$ 

Choosing $\alpha=1$ we defined the $n$th Fej\'er mean, as special case. 
For  Walsh-Paley-Fej\'er  kernel  functions  we  have
$K_n(x)\to 0$, while $n\to \infty$ for  every $x\neq 0$.   However,  they  can  take  negative
values which is a different situation from the trigonometric case. On  the  other  hand,  for  Walsh-Kaczmarz-Fej\'er  kernel  functions
we have $K_n(x)\to\infty $, while $n\to\infty$ at every dyadic rational.
The last fact shows that the behavior of Walsh-Kaczmarz system worse than the 
behavior of Walsh-Paley system in this special sense. See later inequality \eqref{eq-sn}, as well. 

It is well-known that the $2^n$th Dirichlet kernels have a
closed form (see e.g.  \cite{SWSP})
\begin{equation}\label{Dir}
D_{2^n}^w(x)=D_{2^n}^\kappa (x)=D_{2^n} (x)=\begin{cases} 0, & \textrm{if } x\not\in I_n,\\
2^{n},& \textrm{if } x\in I_n.\end{cases}
\end{equation}

The Kronecker product $\left( \psi_{n}:n\in \mathbb{
N}^d\right) $ of $d$ Walsh-(Kaczmarz) system is said to be the
$d$-dimensional (multi-dimensional) Walsh-(Kaczmarz) system. That is,
$$
\psi_{n}\left( x\right) =\psi_{n_1}\left( x^1\right)\dots
\psi_{n_d}\left( x^d\right) ,
$$
where $n=(n_1,\ldots ,n_d)$ and $x=(x^1,\ldots ,x^d)$.

If $f\in L^1\left( G^{d}\right) ,$ then the number $\widehat{f}^\psi\left( n\right) 
:=\int\limits_{G^{d}}f\psi_{n}\quad\left(
n\in \mathbb{N}^d\right) $ is said to be the $n$th Walsh-(Kaczmarz)-Fourier coefficient of $f.$ We can extend this
definition to
martingales in the usual way (see Weisz \cite{we2, we3}). 

The $d$-dimensional $(C,\alpha)$ ($\alpha=(\alpha_{n_1},\ldots ,\alpha_{n_d})$) or Ces\`aro mean of a martingale is defined by 
$$
\sigma_{n}^{\psi,\alpha} f(x):=\frac{1}{\prod_{i=1}^d A_{n_i}^{\alpha_i}}\sum_{j=1}^d\sum_{k_j=0}^{n_j}    
\prod_{i=1}^d A_{n_i-k_i}^{\alpha_i-1}
S_{k}^\psi(f;x).
$$
It is known that 
$$
K_{n}^{\psi,\alpha}(x)=K_{n_1}^{\psi,\alpha_1}(x^1)\dots K_{n_d}^{\psi,\alpha_d}(x^d).
$$

In 1948 $\breve{\text{S}}$neider \cite{Snei} introduced the Walsh-Kaczmarz
system and showed that the inequality
\begin{equation}\label{eq-sn}
\limsup_{n\to \infty }\frac{D_{n}^{\kappa }(x)}{\log n}\geq C>0
\end{equation}
holds a.e. 
This inequality shows a big difference between the two arrangements of  Walsh system.

In 1974 Schipp \cite{Sch2} and  Young \cite{Y} proved
that the Walsh-Kaczmarz system is a convergence system. Skvortsov
in 1981 \cite{Sk1} showed that the Fej\'er means with respect to
the Walsh-Kaczmarz system converge uniformly to $f$ for any
continuous functions $f$. G\'at \cite{gat} proved for any
integrable functions, that the Fej\'er means with respect to the
Walsh-Kaczmarz system converge almost everywhere to the function. 
Moreover, he showed that the maximal operator $\sigma^{\kappa,*}$ of Walsh-Kaczmarz-Fej\'er means 
 is weak type $(1,1)$ and of type $(p,p)$ for all $1<p\leq\infty$.
G\'at's result was generalized by Simon \cite{S2}, 
who showed that the maximal operator $\sigma^{\kappa,*}$ is of type $(H_p,L_p)$ for $p>1/2$.
In the endpoint case $p=1/2$  
Goginava \cite{Gog-PM} proved that the maximal operator $\sigma^{\kappa,*}$ is not of type $(H_{1/2},L_{1/2})$ and
Weisz \cite{We5} showed that the maximal operator is of weak type $(H_{1/2},L_{1/2})$. Recently, the rate of the deviant behaviour in the endpoint $p=1/2$ was discussed by Goginava and the author \cite{gog-nagy}. The case $0<p<1/2$ can be found in the papers of Tephnadze \cite{Tp1,Tp2}.

In 2004, the boundedness of the maximal operator of the Ces\`aro means ($0<\alpha\leq 1$) was investigated by Simon \cite{S3}.  It was showed that the maximal operator is bounded from the Hardy space $H_p$ to the space $L_p$ for $p>p_0:=1/(1+\alpha)$ and of weak type (1,1). Moreover, he showed the inequality
\begin{equation}\label{eq-norm}
\sup_n \Vert K_n^{\kappa,\alpha} \Vert_1< \infty \quad \textrm{ for all } 0<\alpha\leq 1.
\end{equation}
In the endpoint case Goginava showed that the endpoint  $p_0:=1/(1+\alpha)$ is essential. That is, he proved that the maximal operator of Ces\`aro means is not bounded from the Hardy space $H_{p_0}$ to the space $L_{p_0}$ \cite{Gog-Ann}.

For $x=(x^1,x^2,\ldots ,x^d)\in G^d$ and $n=(n_1,n_2,\ldots ,n_d)\in {\mathbb N}^d$ the $d$-dimensional rectangles are defined by $I_n(x):=I_{n_1}(x^1)\times\dots\times I_{n_d}(x^d)$. For $n\in {\mathbb N}^d$ the $\sigma $-algebra generated by the rectangles $\{ I_n(x), x\in G^d\}$ is denoted by ${\mathcal F}_{n}$. The conditional expectation operators relative to ${\mathcal F}_n$ are denoted by $E_n$.

Suppose that for all $j=2,\ldots,d$ the functions $\gamma_j\colon [1,+\infty)\to [1,+\infty)$ 
are strictly monotone increasing continuous functions with properties $\lim_{+\infty} \gamma_j =+\infty$ and $ \gamma_j (1)=1$. Moreover, suppose that there exist $\zeta,c_{j,1},c_{j,2}>1$ such that the inequality 
\begin{equation}\label{CRF}
c_{j,1}\gamma_j(x)\leq \gamma_j(\zeta x)\leq c_{j,2}\gamma_j (x)
\end{equation}
holds for each $x\geq 1.$ In this case the functions $\gamma_j$ are called CRF (cone-like restriction functions). Let $\gamma:=(\gamma_2,\ldots ,\gamma_d)$ and $\beta_j\geq 1$ be fixed ($j=2,\ldots ,d$).
We will investigate the maximal operator of the multi-dimensional $(C,\alpha)$ means and the convergence 
over a cone-like set $L$ (with respect to the first dimension), where 
$$
L:=\{ n\in {\mathbb N}^d: \beta_j^{-1}\gamma_j(n_1)\leq n_j\leq \beta_j \gamma_j(n_1), j=2,\ldots ,d\}.
$$
If each $\gamma_j$ is the identical function then we get a cone. The cone-like sets were introduced by G\'at in dimension two \cite{G3}. The condition \eqref{CRF} on the function $\gamma$ 
is natural, because G\'at \cite{G3} proved that to each cone-like set with respect to 
the first dimension there exists a larger cone-like set with respect to the second dimension and reversely, if and only if the inequality \eqref{CRF} holds.

Weisz defined a new type martingale Hardy space depending on the function $\gamma$ (see \cite{we-cone}). For a given 
$n_1\in{\mathbb N}$ set $n_j:=|\gamma_j(2^{n_1})|$ ($j=2,\ldots ,d$), that is, $n_j$ is the order of $\gamma_j(2^{n_1})$ 
(this means that $2^{n_j}\leq \gamma_j(2^{n_1})<2^{n_j+1}$ for $j=2,\ldots ,d$). Let $\overline{n}_1:=(n_1,\ldots ,n_d)$. Since, the functions $\gamma_j$ are increasing, the sequence $(\overline{n}_1,\ n_1\in{\mathbb N})$ is increasing, too.
It is given a class of one-parameter martingales $f=(f_{\overline{n}_1},\ n_1\in{\mathbb N})$ with respect to the $\sigma$-algebras $({\mathcal F}_{\overline{n_1}},\  n_1\in{\mathbb N})$. The maximal function of a martingale $f$ is defined by
$
f^{*}=\sup\limits_{n_1\in \mathbb{N}}\left| f_{ \overline{n}_1 }\right| .
$
For $0<p\leq \infty $ the  martingale Hardy space $H_{p}^\gamma (G^d)$ consists of all
martingales for which
$
\left\| f\right\| _{H_{p}^\gamma}:=\left\| f^{*}\right\| _{p}<\infty .
$
It is known (see \cite{we3}) that $H_p^\gamma \sim L_p$ for $1<p\leq \infty$, where $\sim$ denotes the equivalence of the norms and spaces.

If $f\in L_1(G^d)$ then it is easy to show that the sequence $(S_{2^{n_1},...,2^{n_d}}(f): \overline{n_1}=(n_1,...,n_d), n_1\in{\mathbb N})$ is a one-parameter martingale with respect to the $\sigma$-algebras $({\mathcal F}_{\overline{n_1}},\  n_1\in{\mathbb N})$. In this case the maximal function can also be given by 
$$
f^*(x)=\sup_{n_1\in {\mathbb N}} \frac{1}{\mes (I_{\overline{n_1}}(x))} \left|
\int_{ I_{\overline{n_1}}(x)}f(u) d\mu(u)
\right|=\sup_{n_1\in {\mathbb N}} |S_{2^{n_1},...,2^{n_d}}(f,x)|
$$
for $x\in G^d$.

We define the maximal operator $\sigma_L^{\kappa,\alpha,*}$ by 
$$
\sigma_L^{\kappa,\alpha,*}f(x):=\sup_{n\in L} |\sigma_n^{\kappa,\alpha} f (x)|.
$$

For double Walsh-Paley-Fourier series, M\'oricz, Schipp and Wade \cite{MSW} proved  that $\sigma_n f$ converge to $f$ a.e. in the Pringsheim sense (that is, no restriction on the indices other than $\min (n_1,n_2)\to\infty$) for all 
functions $f\in L\log^+ L$. 
The a.e. convergence of  Fej\'er means $\sigma_n f$ of  
integrable functions, where the set of indices is inside a positive cone around the identical function, that is $\beta^{-1}\leq n_1/n_2\leq \beta$ is provided with some fixed parameter $\beta\geq 1$, was proved by G\'at \cite{G} and Weisz \cite{We}. A common generalization of results of M\'oricz, Schipp, Wade \cite{MSW} and G\'at \cite{G},  Weisz \cite{We} for cone-like set was given by the first author  and G\'at in \cite{G-N}. Namely,  a necessary and 
sufficient condition for cone-like sets in order to preserve the convergence property, was given. The trigonometric case was treated by G\'at \cite{G3}.

Connecting to the original paper \cite{G3} on trigonometric system, G\'at asked the following. 
What could we state for other systems for example Walsh-Paley, Walsh-Kaczmarz and Vilenkin systems and for 
other means for example logarithmic means, Riesz means, (C,$\alpha$) means? Some parts of G\'at's question was answered by Weisz \cite{we-cone}, Blahota and the first author \cite{B-N, N-GMJ} and naturally in paper \cite{G-N}.

In 2011, the properties of the maximal operator of the $(C,\alpha)$ and Riesz means of a multi-dimensional Vilenkin-Fourier series  provided that 
the supremum in the maximal operator is taken over a cone-like set, 
was discussed by Weisz \cite{we-cone}. Namely, it was proved that the maximal operator is bounded from dyadic Hardy space $H_p$ 
to the space $L_p$ for $p_0< p \leq\infty$ ($p_0:=\max \{ 1/(1+\alpha_k): k=1,\ldots ,d\}$) and is of weak type $(1,1)$. Recently, it was showed that the index $p_0$ is sharp. Namely, it was proven that the maximal operator is not bounded from the dyadic Hardy space $H_{p_0}$ 
to the space $L_{p_0}$ \cite{B-N}.   Detailed list of the reached results for one- and several dimensional Walsh-like systems can be found in \cite{We-amapn}.

For the two-dimensional Walsh-Kaczmarz-Fourier series Simon proved \cite{S1} that the  cone-restricted maximal operator 
of the Fej\'er means is bounded from the Hardy space $H_p$ 
to the space $L_p$ for all $1/2< p$ (here the set of indices is inside a positive cone around the identical function). That is, the a.e. convergence of cone-restricted Fej\'er means holds for the 
Walsh-Kaczmarz system, as well. 
Moreover, it was proved that $p=1/2$ is essential.  In 2007, Goginava 
and the first author proved that the cone-restricted maximal operator is not bounded from the Hardy space $H_{1/2}$ to the space weak-$L_{1/2}$ \cite{gog-N}. The cone-like restricted two-dimensional maximal operator of Fej\'er means was discussed in \cite{N-GMJ}.

Motivated by the works of Weisz \cite{we-cone}, Simon \cite{S1} and the above mentioned question of G\'at we prove that 
 the maximal operator $\sigma^{\kappa,\alpha,*}_L$ is bounded from the dyadic Hardy space $H_p^\gamma$ 
to the space $L_p$ for $p_0< p $ and is of weak type $(1,1)$. As a corollary we get the 
theorem of Simon \cite{S1} on the a.e. convergence of cone-restricted Fej\'er means. 
At the endpoint $p=p_0$, we show that the maximal operator $\sigma^{\kappa,\alpha,*}_L$ is not bounded from the Hardy space $H_{p_0}^\gamma$ to the space $L_{p_0}$. 

In dimension 2, the case $\alpha_1=\alpha_2=1$ was discussed in \cite{N-GMJ}. Unfortunately, the counterexample martingale and the method is not suitable for 
case $0<\alpha_i<1$ ($i=1,\ldots, d$).  

\section{Auxiliary propositions and main results}

First, we formulate our main theorems.

\begin{theorem}\label{thm1}Let $\gamma$ be CRF. 
The maximal operator ${\sigma}^{\kappa,\alpha,*}_L$  is  bounded from the Hardy space $H_{p}^\gamma$ to the
space $L_{p}$ for $p_0<p\leq 1$ ($p_0:=\max\{1/(1+\alpha_i); i=1,\ldots ,d\}$).
\end{theorem}
By standard argument we have that, if $1<p\leq \infty$ then ${\sigma}^{\kappa,\alpha,*}_L$ is of type $(p,p)$ and 
of weak type $(1,1)$.

\begin{theorem}\label{thm2}
Let $\gamma$ be CRF. Then for any $f\in L^1$ 
$$
\lim_{\substack{\land n\to \infty \\ n\in L}} \sigma_n^{\kappa,\alpha} f=f
$$
holds almost everywhere.
\end{theorem}

We immediately have the theorem of Simon  \cite{S1} as a corollary.
\begin{corollary}[Simon \cite{S1}]
Let $f\in L^1$ and $\beta \geq 1$ be a fixed parameter. Then 
$$
\lim_{\substack{\land n\to \infty \\ \beta^{-1}\leq n_1/n_2\leq \beta}} \sigma_n^\kappa f=f
$$ 
holds a.e.
\end{corollary}

\begin{theorem}
\label{thm3}Let $\gamma$ be CRF and $\alpha_1\leq \ldots \leq \alpha_d$. 
The maximal operator ${\sigma}^{\kappa,\alpha,*}_L$  is  not bounded from the Hardy space $H_{p_0}^\gamma$ to the
space $L_{p_0}$.
\end{theorem}

To prove our Theorem \ref{thm1}, \ref{thm2}, \ref{thm3} we need the following Lemma of Weisz \cite{we3}, the concept of the atoms (for more details see \cite{we-cone}) and a  Lemma of Goginava \cite{Gog-Ann}.

A bounded measurable function $a$ is a $p$-atom, if there exists a dyadic $d$-dimensional rectangle $I\in{\mathcal F}_{\overline{n}_1}$, such that
\begin{itemize}
\item[a)]  $\supp a\subseteq I$,
\item[b)] $\Vert a\Vert_\infty \leq \mu(I)^{-1/p}$,
\item[c)] $\int_{I}a d\mu =0$.
\end{itemize}
\begin{lemma}[Weisz \cite{we3}]\label{lemma-weisz}
Suppose that the operator $T$ is $\sigma$-sublinear and $p$-quasilocal for any $0<p< 1$. If $T$ is bounded from $L_\infty$ to $L_\infty $, then
$$
\Vert Tf\Vert_p\leq c_p \Vert f\Vert_{H_p} \quad
\textrm{ for all } f\in H_p .
$$
\end{lemma}
\begin{lemma}[Goginava \cite{Gog-Ann}]\label{Lemma-Gog}
Let $n\in {\mathbb N}$ and $0<\alpha\leq 1$. Then 
$$
\int_G \max_{1\leq N <2^n}(A_{N-1}^\alpha|K_N^\alpha (x)|)^{1/(1+\alpha)}d\mu(x)\geq c(\alpha)
\frac{n}{\log (n+2)}.
$$
\end{lemma}

\section{Proofs of the theorems}
Now, we prove our main Theorems.
\begin{proof}[Proof of Theorem \ref{thm1}] 
Using   Lemma \ref{lemma-weisz} of Weisz we have to prove that  
the operator 
${\sigma}^{\kappa,\alpha, *}_L$
is bounded from the space $L_\infty$ to the space $L_\infty$. It immediately follows from the inequality 
\eqref{eq-norm}. 

Let $a$ be a $p$-atom, with  support $I$. We can assume that  $I=I_{N_1}\times\ldots \times  I_{N_d}$ (with $2^{N_j}\leq \gamma_j(2^{N_1})<2^{N_j+1}$, $j=2,\ldots , d$), $\Vert a\Vert_\infty \leq 2^{(N_1+\ldots +N_d)/p}$ and 
$\int_{I}a d\mu =0$.

Set $\delta:=\max\{ \zeta^{\log_{c_{j,1}}2\beta_j+1}: j=2,\ldots,d\}.$ 
If $n_1\leq 2^{N_1}/\delta$, then

\begin{eqnarray*}
n_j&\leq& \beta_j \gamma_j(n_1)\leq \beta_j \gamma_j (2^{N_1}\zeta^{-\log_{c_{j,1}}2\beta_j-1} )\\ 
&\leq& \beta_j \frac{1}{c_{j,1}^{\log_{c_{j,1}}2\beta_j+1}}\gamma_j(2^{N_1})
\leq  \frac{\gamma_j(2^{ N_1})}{2}\leq 2^{N_j}.
\end{eqnarray*}

$\zeta,c_{j,1},c_{j,2}>1$, $\beta_j\geq 1$ imply $n_1< 2^{N_1}$ and $n_j\leq \gamma_j(2^{N_1})/2< 2^{N_j}$ ($j=2,\ldots,d$). In this case  the $(m_1,\ldots, m_d)$-th Fourier coefficients are zeros for  
$m_1\leq n_1, \ldots ,m_d\leq n_d$. This gives 
$\sigma_n f=0$ ($n=(n_1,\ldots,n_d)$).

That is, we could suppose that $n_1>2^{N_1}/\delta.$ 
This yields that
$$
n_j\geq \frac{\gamma_j (n_1)}{\beta_j}
\geq 
\frac{\gamma_j(2^{N_1}/\delta)}{\beta_j}
\geq \frac{1}{\beta_j c_{j,2}^{\max\{\log_{c_{j,1}}2\beta_j+1: j=2,\ldots,d\}}}
\gamma_j(2^{N_1})
\geq \frac{\gamma_j(2^{N_1})}{\delta_j'}\geq \frac{2^{N_j}}{\delta'}
$$
with $\delta':=\max_{j=2,\ldots, d}\delta_j'$ for all $j=2,\ldots , d$.  $\delta'>1$ can be assumed.

The proof will be complete, if we show that 
the maximal operator ${\sigma}^{\kappa,   \alpha,*}_L$ is  $p$-quasilocal for $p_0<p\leq 1$. That is,
there exists a constant $c_p$ such that the inequality 
$$
\int_{\overline{I}} |{\sigma}^{\kappa,\alpha,*}_L a
|^{p}d\mu \leq c_p<\infty
$$
holds for all atom $a$ in $H_p^\gamma$ with support $I=I_{N_1}\times \ldots \times I_{N_d}$ (with $\overline{N_1}=(N_1, \ldots ,N_d)$). 
It is well known that the concept of $p$-quasi-locality of the maximal operator ${\sigma}^{\kappa,\alpha,*}$ can be modified as follows \cite{S1}: there exists $r=0,1,\ldots $, such that
\begin{equation}
\int_{\overline{I^r}} |{\sigma}^{\kappa,\alpha,*}_L a
|^{p}d\mu \leq c_p<\infty,
\end{equation}
where $I^r:= I_{N_1}^r\times \ldots \times I_{N_d}^r
:=I_{N_1-r}\times \ldots \times I_{N_d-r}$ ($N_j-r\geq 0$ for all $j=1,\ldots,d$). We will give the value of $r$ later.

Let us set $x=(x^1,\ldots,x^d)\in \overline{I^r}$. 
\[
\begin{split}
|\sigma_n^{\kappa,\alpha} a(x)|&=
\left|\int_{I} a(t^1,\ldots,t^d)K_{n_1}^{\kappa,\alpha_1}(x^1+t^1)\ldots K_{n_d}^{\kappa,\alpha_d}(x^d+t^d)d\mu(t)\right| \\
&\leq  2^{(N_1+\ldots+N_d)/p}\int_{I_{N_1}} |K_{n_1}^{\kappa,\alpha_1}(x^1+t^1)|d\mu(t^1)\ldots \int_{I_{N_d}}|K_{n_d}^{\kappa,\alpha_d}(x^d+t^d)|d\mu(t^d)
\end{split}
\]
Now, we decompose the set $\overline{I^r}=\overline{I_{N_1}^r\times\ldots \times I_{N_d}^r}$ as the following disjoint union
\[
\begin{split}
\overline{I^r}=&(\overline{I_{N_1}^r}\times \ldots \times \overline{I_{N_d}^r})\cup\\
&\cup ({I_{N_1}^r}\times \overline{I_{N_2}^r}\times \ldots \times \overline{I_{N_d}^r})\cup \ldots \cup 
(\overline{I_{N_1}^r}\times  \ldots \times\overline{I_{N_{d-1}}^r}\times {I_{N_d}^r})\cup\\
&\vdots\\
&\cup 
(\overline{I_{N_1}^r}\times {I_{N_2}^r}\times \ldots \times {I_{N_d}^r})\cup \ldots \cup 
({I_{N_1}^r}\times  \ldots \times{I_{N_{d-1}}^r}\times \overline{I_{N_d}^r}).
\end{split}
\]
Let us set $\delta'':=\max\{ \delta,\delta'\}$, set $r\in \mathbb P$ such that 
$2^{-r}\leq 1/\delta''\leq 2^{-r+1}$, and 
$L^{r,l}:={I_{N_1}^r}\times  \ldots \times{I_{N_{l}}^r}\times \overline{I_{N_{l+1}}^r} \times \ldots\times \overline{I_{N_d}^r}$ for $l=0,\ldots,d$. 
We define 
$$J_i:= \int_{{I_{N_i}^r}}\left( \sup_{n_i\geq 2^{N_i}/\delta''}\int_{I_{N_i}} |K_{n_i}^{\kappa,\alpha_i}(x^i+t^i)|d\mu(t^i)\right)^p\mu(x^i),\quad i=1,\ldots, l,$$
$$\overline{J_j}:= \int_{\overline{I_{N_j}^r}}\left( \sup_{n_j\geq 2^{N_j}/\delta''}\int_{I_{N_j}} |K_{n_j}^{\kappa,\alpha_j}(x^j+t^j)|d\mu(t^j)\right)^p\mu(x^j),\quad j=l+1,\ldots, d.$$
Now, we get 
$$\int_{L^{r,l}} |{\sigma}^{\kappa,\alpha,*}_L a
|^{p}d\mu\leq 2^{N_1+\ldots+N_d}
J_1 \cdot \ldots \cdot J_l\cdot \overline{J_{l+1}}\cdot \ldots \cdot 
\overline{J_d}.
$$

First, we discuss the integrals $J_i$ ($i=1,\ldots, l$). Inequality \eqref{eq-norm} and the definitions of $\delta'',r$ immediately yield that 
\begin{equation}\label{J-i}
J_i \leq   2^{-(N_i-r)}\left( \sup_{n_1\in {\mathbb N}} \Vert K_{n_i}^{\kappa,\alpha_i}\Vert_1\right)^p\leq c_p 2^{-N_i}.
\end{equation}

Second, we discuss the integrals $\overline{J_j}$ ($j=l+1,\ldots, d$).
In the paper \cite[48-59 pages]{S3}, Simon showed that 
\begin{equation}\label{eq-Simon}
\int_{\overline{I_{N}}}\left( \sup_{n\geq 2^{N}}\int_{I_N} |K_{n}^{\kappa,\alpha}(x+t)|d\mu(t)\right)^p\mu(x)
\leq c_p2^{-N}\quad \textrm{ if } p>\frac{1}{1+\alpha}\ (0<\alpha\leq 1).
\end{equation}
Using this we write 
\begin{equation}\label{J-j}
\overline{J_j} \leq
\int_{\overline{I_{N_j}^r}}\left( \sup_{n_j\geq 2^{N_j-r}}\int_{I_{N_j}^r} |K_{n_j}^{\kappa,\alpha_j}(x^j+t^j)|d\mu(t^j)\right)^p\mu(x^j)\leq c_p 2^{-N_j}
\quad \textrm{if }p>\frac{1}{1+\alpha_j}.
\end{equation}
Inequalities \eqref{J-i}, \eqref{J-j} yield 
$$\int_{L^{r,l}} |{\sigma}^{\kappa,\alpha,*}_L a
|^{p}d\mu\leq c_p\quad \textrm{for all $l$ if }p>p_0.
$$
The decomposition of $\overline{I^r}$ gives
$$
\int_{\overline{I^r}} |{\sigma}^{\kappa,\alpha,*}_L a
|^{p}d\mu\leq c_{p,d}\quad \textrm{if }p>p_0.
$$
This completes the proof of Theorem \ref{thm1}.
\end{proof}

The weak type $(1,1)$ inequality follows by interpolation.
The set of Walsh-Kaczmarz polynomials is dense in $L_1$. The weak-type (1,1) inequality and the usual density argument imply Theorem \ref{thm2}.

\begin{proof}[Proof of Theorem \ref{thm3}] In the present proof we use a  counterexample martingale. 
Let us set
$$
f_{\overline{n_1}}(x):=(D_{2^{n_1+1}}(x^1)-D_{2^{n_1}}(x^1))\prod_{j=2}^{d}w_{2^{n_j}-1}(x^j) 
,
$$ where $n_2,\ldots ,n_d$ is defined  to $n_1$, earlier.

Now, we calculate $S_{j}^\kappa (f_{\overline{n_1}};x)$. 
Since, $\omega_{2^n-1}(x)=r_{n-1}(x) (-1)^{\sum_{i=0}^{n-2}x_i}=\kappa_{2^n-1}(x)$, we have
$$
\hat{f}_{\overline{n_1}}^\kappa (k)=\begin{cases}
1,& \textrm{ if } k_1=2^{n_1},\ldots ,2^{n_1+1}-1, \textrm{ and } k_j=2^{n_j}-1 \textrm{ for all } j=2,\ldots ,d; \\
0,& \textrm{ otherwise.}
             \end{cases}
$$
\begin{eqnarray}\label{S-j}
S_j^\kappa (f_{\overline{n_1}},x)&=&\sum_{\nu=0}^{j_1-1}\hat{f}_{\overline{n_1}}^\kappa (\nu, 2^{n_2}-1,...,2^{n_{d}}-1)\omega_\nu(x^1)\prod_{l=2}^{d}w_{2^{n_l}-1}(x^l)\nonumber\\
&=&\begin{cases}
(D_{j_1}(x^1)-D_{2^{n_1}}(x^1))\prod_{l=2}^{d}w_{2^{n_l}-1}(x^l)& \textrm{if } 
j_1=2^{n_1}+1,...,2^{n_1+1}-1, \textrm{ and}\\
 & j_l\geq 2^{n_l} \textrm{ for all } l=2,...,d;\\
f_{\overline{n_1}}(x)& \textrm{if }j_1\geq 2^{n_1+1} \textrm{ and }j_l\geq 2^{n_l} \\ 
 & \textrm{for all } l=2,...,d;\\
0& \textrm{otherwise.}
\end{cases}
\end{eqnarray}
We immediately have that 
$$
f_{\overline{n_1}}^*(x)=\sup_{m_1\in {\mathbb N}}|S_{2^{m_1},\ldots, 2^{m_d}}(f_{\overline{n_1}},x)|=|f_{\overline{n_1}}(x)|,
$$
where $\overline{m_1}=(m_1,\ldots , m_d)$.
\begin{equation}\label{norm}
\Vert f_{\overline{n_1}}\Vert_{H_{p_0}^\gamma}=\Vert f_{\overline{n_1}}^*\Vert_{p_0}=\Vert D_{2^{n_1}}\Vert_{p_0}=2^{(1-1/p_0)n_1} <\infty.
\end{equation}
That is $f_{\overline{n_1}}\in H_{p_0}^\gamma$.

We can write the $n$th  Dirichlet kernel with respect to the Walsh-Kaczmarz system in the following form:
\begin{eqnarray}\label{eq-mag}
D_n^\kappa (x)&=&D_{2^{\vert n\vert}}(x)
+\sum_{k=2^{\vert n\vert}}^{n-1}r_{\vert k\vert}(x)w_{k-2^{|n|}}(\tau_{\vert k\vert}(x))\nonumber \\
&=&D_{2^{\vert n\vert}}(x)+r_{\vert n\vert}(x)
D_{n-2^{\vert n\vert}}^w(\tau_{\vert n\vert}(x))
\end{eqnarray}
We set $L_1^N:=2^{n_1}+N,$ where $0<N<2^{n_1}$ and 
$L_j^N:= [\gamma_j (2^{n_1}+N)]$ for $j=2,\ldots,d$ (where $[x]$ denotes the integer part of $x$). In this case $L^N:=(L_1^N,\ldots ,L_d^N)\in L$. 
Let us calculate $\sigma_{L^N}^{\kappa,\alpha} f_{\overline{n_1}}$. 

By inequalities \eqref{eq-mag} and \eqref{S-j} we get 
\begin{eqnarray*}
|\sigma_{L^N}^{\kappa,\alpha} f_{\overline{n_1}} (x)|&= &\frac{1}{\prod_{j=1}^d A_{L_j^N}^{\alpha_j}}\left|\sum_{j=1}^d\sum_{k_j=0}^{L_j^N}\prod_{i=1}^d A_{L_i^N-k_i}^{\alpha_i-1}
S_{k}^\kappa (f_{\overline{n_1}};x)\right|\\
&=&
\frac{1}{\prod_{j=1}^{d} A_{L_j^N}^{\alpha_j}}\left|\sum_{j=2}^{d}\sum_{k_j=2^{n_j}}^{L_j^N}
\sum_{k_1=2^{n_1}+1}^{L_1^N}
\prod_{i=1}^d A_{L_i^N-k_i}^{\alpha_i-1}
S_{k}^\kappa (f_{\overline{n_1}};x)
\right|\\
&=&
\frac{1}{\prod_{j=1}^{d} A_{L_j^N}^{\alpha_j}}\left|\sum_{j=2}^{d}\sum_{k_j=2^{n_j}}^{L_j^N}
\sum_{k_1=2^{n_1}+1}^{L_1^N}
\prod_{i=1}^d A_{L_i^N-k_i}^{\alpha_i-1}
\prod_{l=2}^{d}w_{2^{n_l}-1}(x^l)(D_{k_1}^\kappa(x^1)-D_{2^{n_1}}(x^1))\right|\\
&=&
\frac{1}{\prod_{j=1}^d A_{L_j^N}^{\alpha_j}}\left|\sum_{j=2}^{d}
\sum_{k_j=1}^{L_j^N-2^{n_j}}
\prod_{i=2}^{d} A_{L_i^N-2^{n_i}-k_i}^{\alpha_i-1}
\sum_{k_1=1}^{L_1^N-2^{n_1}}A_{L_1^N-2^{n_1}-k_1}^{\alpha_1-1}
D_{k_1}^w(\tau_{n_1}(x^1))
\right|\\
&=&\frac{1}{\prod_{j=1}^d A_{L_j^N}^{\alpha_j}}
\left|\sum_{j=2}^{d}
\sum_{k_j=1}^{L_j^N-2^{n_j}}
\prod_{i=2}^{d} A_{L_i^N-2^{n_i}-k_i}^{\alpha_i-1}
\right|
 \left| A_{L_1^N-2^{n_1}}^{\alpha_1}K_{L_1^N-2^{n_1}}^{w,\alpha_1}(\tau_{n_1}(x^1))\right|\\
 &\geq& \frac{c(\alpha)}{2^{n_1\alpha_1}}A_N^{\alpha_1}|K_N^{w,\alpha_1}(\tau_{n_1}(x^1))|.
\end{eqnarray*}
Now, we write that 
\begin{equation*}
\sigma_L^{\kappa,\alpha,*}f_{\overline{n_1}} (x)
\geq \max_{1\leq N< 2^{n_1}}|\sigma_{L^N}^{\kappa,\alpha} f_{\overline{n_1}} (x)|
\geq \frac{c(\alpha)}{2^{n_1\alpha_1}}
\max_{1\leq N< 2^{n_1}}
A_N^{\alpha_1}|K_N^{w,\alpha_1}(\tau_{n_1}(x^1))|.
\end{equation*}
Inequality \eqref{norm} and Lemma \ref{Lemma-Gog} yield 
\begin{eqnarray*}
\frac{\Vert \sigma_L^{\kappa,\alpha,*}f_{\overline{n_1}}\Vert _{p_0}}{\Vert f_{\overline{n_1}}\Vert_{H_{p_0}^\gamma}}
&\geq& \frac{1}{2^{(1-1/p_0)n_1}} \left( 
\int_{G^d}\max_{1\leq N< 2^{n_1}}|\sigma_{L^N}^{\kappa,\alpha} f_{\overline{n_1}} (x)|^{p_0}d\mu(x)
\right)^{1/p_0}\\
&\geq& \frac{c(\alpha)2^{n_1\alpha_1}}{2^{n_1\alpha_1}}
\left( 
\int_{G}\max_{1\leq N< 2^{n_1}}
A_N^{\alpha_1}|K_N^{w,\alpha_1}(\tau_{n_1}(x^1))|
d\mu(x^1)
\right)^{1/p_0}\\
&=&\frac{c(\alpha)2^{n_1\alpha_1}}{2^{n_1\alpha_1}}
\left( 
\int_{G}\max_{1\leq N< 2^{n_1}}
A_N^{\alpha_1}|K_N^{w,\alpha_1}(x^1)|
d\mu(x^1)
\right)^{1/p_0}\\
&\geq& c(\alpha)\left( \frac{n_1}{\log n_1}
\right)^{1+\alpha_1}\to\infty  \quad \textrm{while } n_1\to \infty.
\end{eqnarray*}
This completes the proof of Theorem \ref{thm3}.
\end{proof}

\end{document}